\documentclass[11pt]{amsart}

\usepackage[shortalphabetic]{amsrefs}
\usepackage{eucal,graphicx,amssymb,mathrsfs}
\usepackage[all, knot]{xy}

\theoremstyle{definition}

\newtheorem*{ex*}{Example}

\theoremstyle{plain}
\newtheorem{thm}{Theorem}[section]
\newtheorem{prop}[thm]{Proposition}
\newtheorem*{cor*}{Corollary}

\newtheorem{lem}[thm]{Lemma}

\newtheorem{thm*}{Theorem}

\newtheorem*{thmA}{Theorem A}

\theoremstyle{remark}
\newtheorem{rem}{Remark}

\numberwithin{equation}{section}

\newcommand{\origin}{\mathbf{0}}

\DeclareMathOperator{\id}{id}

\DeclareMathOperator{\supp}{supp}

\DeclareMathOperator{\SL}{SL}

\def\C{\mathbb{C}}

\def\P{\mathbb{P}}
\def\R{\mathbb{R}}
\def\Z{\mathbb{Z}}
\def\P{\mathbb{P}}

\def\S{\mathbb{S}}

\def\N{\mathbb{N}}
\def\TT{\mathbf{T}}

\def\RR{{\mathcal{R}}}

\def\fan{\Delta}

\def\n0{{\bf n_0}}

\title{Rotation Numbers of Elements in Thompson's Group ${\bf T}$}

\author{Jeffrey Diller}
\address{Department of Mathematics\\
         University of Notre Dame\\
         Notre Dame, IN 46556}
\email{diller.1@nd.edu}

\author{Jan-Li Lin}
\address{Department of Mathematics\\
         Northwestern University\\
         Evanston, IL 60208}
\email{janlin@math.northwestern.edu}

\thanks{This work was supported in part by National Science Foundation grant DMS--1066978.}
\begin{document}

\begin{abstract}
We give a simple combinatorial proof that the rotation number for each element in Thompson's group ${\bf T}$
is rational.
\end{abstract}

\maketitle

\section{Introduction}
\label{introduction}

In 1965, Richard Thompson defined three groups which have furnished counterexamples to various conjectures in group theory.  One of these $\TT$, which will be the subject of this article, is the group of `dyadic' circle homeomorphisms $f:\S^1\to \S^1$.  That is, if one takes $\S^1$ to be the interval $[0,1]$ with endpoints identified, then
\begin{itemize}
\item $f$ preserves the set of dyadic rational numbers (i.e., numbers of the form $p\cdot2^q$, $p,q\in\Z$),
\item $f$ is linear except at a finite number of dyadic rational points,
\item on each interval such that $f$ is linear, the derivative (i.e., slope) is a power of $2$.
\end{itemize}
For more information on Thompson's groups, see~\cites{CFP, Ghys, GS}. (Notice that the group ${\bf T}$ is denoted by $G$ in \cites{Ghys, GS}.)

The \emph{rotation number} \cite{KH} of an orientation preserving circle homeomorphism $f:\S^1\to \S^1$ is the quantity
\[
\rho(f)=\lim_{n\to\infty}\frac{\widetilde{f}^n(x)}{n} \mod \Z,
\]
where $\widetilde{f}:\R\to\R$ is any lift of $f$ to the universal cover $\R$, and $x\in\R$ is any initial point.  Modulo $\Z$, the limit does not
depend on the choices of $\widetilde{f}$ and $x$, and it is an important dynamical invariant for $f$.  In particular $\rho(f)$ is rational if and only if $f$ has a periodic point.   In 1987, E. Ghys and V. Sergiescu~\cite{GS} proved

\begin{thmA}\mbox{}
\begin{itemize}
\item[I.]
For every element $f$ in Thompson's group ${\bf T}$, the rotation number $\rho(f)$ is rational.
\item[II.]
For every rational number $r$, there is an element $f\in{\bf T}$ such that $\rho(f)=r\mod\Z$.
\end{itemize}
\end{thmA}

\noindent However, as Ghys pointed out in \cite{Ghys}, their proof is very indirect, and there is a need for a better proof.  
Several simpler proofs were given later, in articles by Liousse \cite{Li} and Calegari~\cite{Cal} and the thesis of Matucci~\cite{Ma}. An approach via revealing pairs, implicit in the article \cite{BBetc}, was discussed at length during a 2007 Luminy talk by Bleak. 
The goal of this paper is to provide another elementary combinatorial proof of their result.
It is similar in many respects to those of Calegari and Matucci, but ours uses a different description of $\mathbf{T}$ as the group of piecewise linear automorphisms of $\Z^2$, and where Calegari uses Thurston's notion of a `train track' ours uses combinatorial ideas connected with `fans', i.e. partitions of $\Z^2\subset\R^2$ into rational convex cones. Like his proof, ours provides an effective way to compute rotation numbers for elements in ${\bf T}$.

Fans in $\R^2$ arise naturally in the study of toric surfaces, and the elements of $\bf T$ arise in particular (see~\cites{Us,Fav}) as ``tropicalizations" of plane birational maps preserving the $(\C^*)^2$-invariant two form $\frac{dx\wedge dy}{xy}$.  The ideas in this paper descend from our recent work~\cite{DL}*{Theorem E} concerning dynamics of such maps, and an older result~\cite{DF}*{Theorem 0.1} of Favre and the first author.  Nevertheless, the presentation here is purely combinatorial, based on an analysis of how elements of $\bf T$ act on fans and their refinements, and it makes no appeal to algebraic geometry.

The rest of the paper proceeds as follows. We introduce terminology and notation associated to cones, fans, and  piecewise linear automorphisms of $\Z^2$ in Section~\ref{Sec:PL_auto}, concluding with a description of the relationship between piecewise linear automorphisms and dyadic circle homeomorphisms.  In Section~\ref{Sec:decomposition} we establish a key result (Proposition \ref{prop:decomposition}) about decomposing piecewise linear automorphisms. Section~\ref{Sec:proof} then concludes the proof of Theorem A.

We are indebted to Charles Favre for pointing out that while the Theorem A is a corollary of results in our earlier paper~\cite{DL}, one can extract from that source a direct proof involving no algebraic geometry.  We would also like to thank Victor Kleptsyn for informing us about the papers \cite{Li} and (especially) \cite{Cal} mentioned above.  Finally, we would like to thank the referee for a careful reading and for correcting many notational inconsistencies and typos.

\section{Piecewise Linear Automorphisms of $\Z^2$}
\label{Sec:PL_auto}

Let $\R_+:=\{x\in\R\,|\,x\ge 0\}$ denote the set of non-negative real numbers. A {\em ray} in $\R^2$ is a rational one dimensional cone, i.e. a set of the form $\tau=\R_+ v$ for some $v\in\Z^2\setminus\{\origin\}$.  The first lattice point $v\in\tau\cap(\Z^2\setminus\{\origin\})$ is called the \emph{generator} of $\tau$.
We will call one dimensional cones with (possibly) irrational slope \emph{real rays}.  A {\em sector} is a rational two dimensional cone, i.e. a convex set
$\sigma = \R_+ v_1 + \R_+ v_2$ for some $v_1,v_2\in\Z^2\setminus\{\origin\}$ linearly independent over $\R$.  The rays $\tau_i = \R_+ v_i$, $i=1,2$, bounding $\sigma$ are called {\em facets} of $\sigma$.  The sector $\sigma$ is \emph{regular} if the generators $v_1,v_2$ of the facets of $\sigma$ form a basis for $\Z^2$, i.e. if $|\det[v_1\ v_2]|= 1$.

A \emph{fan} $\fan$ is a set of cones consisting of the zero-dimensional cone $\{\origin\}$, a finite sequence of rays $\tau_0,\tau_1,\cdots,\tau_d$ given in counterclockwise order, and the intervening sectors $\sigma_j$ bounded by $\tau_{j-1}$ and $\tau_j$, $1\leq j\leq d$.  The support of $\fan$, denoted $\supp(\fan)$, is the union of all cones in $\fan$. We assume throughout that our fans are \emph{complete}, i.e. i.e., the sectors $\sigma_j$ in $\Delta$ cover $\R^2$ (hence $\tau_0 = \tau_d$).  A fan is {\em regular} if all its sectors are regular.

%
%


A continuous map $F:\R^2\to\R^2$ is {\em piecewise linear} if there exists a fan $\fan$ with sectors $\sigma_1,\cdots,\sigma_d$ and linear transformations $L_i:\R^2\to\R^2$, $i=1,\cdots,d$ such that $F|_{\sigma_i}=L_i|_{\sigma_i}$ for all $i=1,\cdots,d$.  Any fan satisfying the above condition is said to be {\em compatible} with $F$.
The map $F$ is orientation preserving if and only if each of $L_i$ is.
A {\em piecewise linear automorphism of $\Z^2$} is an orientation preserving,
piecewise linear homeomorphism $F:\R^2\to\R^2$ such that
$F(\Z^2)=\Z^2$. If $F$ is a piecewise linear automorphism, then all the maps $L_i$ in the definition must have $\det(L_i)=1$, i.e., $L_i\in \SL(2,\Z)$.  It follows that if $\Delta$ is a regular fan compatible with $F$, then the fan $F(\Delta)$ obtained by mapping forward all cones in $\Delta$ is also regular, though it is not necessarily compatible with $F$.

%
%

\subsection{Equivalence between realizations of $\TT$}

In proving Theorem A we will always think of elements of $\TT$ as piecewise linear automorphisms of $\Z^2$.  Nevertheless, since $\TT$ is more commonly given as the group of dyadic circle homeomorphisms, we digress briefly to indicate how to translate between the two points of view.  Every piecewise linear homeomorphism $F:\R^2\to\R^2$ induces a circle homeomorphism by projectivizing.  More precisely, $F$ permutes the set $\RR$ of real rays in $\R^2$, which are parameterized by $\S^1$.  If, moreover, $F$ restricts to a bijection $\Z^2\to\Z^2$, then the induced homeomorphism of $\S^1$ completely determines $F$.  It remains to give a homeomorphism $\phi:\RR \to \S^1$ that conjugates (the projectivizations of) piecewise linear automorphisms of $\Z^2$ to dyadic circle homeomorphisms.  In the following, we will regard $\S^1$ as the interval $[0,1]$ with endpoints identified and define a map $\phi:\RR \to [0,1]$. 

This is accomplished locally as follows.  Let $I = [\frac{a}{2^k},\frac{a+1}{2^k}]$ be a ``dyadic standard\footnote{A dyadic sub-interval $I\subset[0,1]$ is called {\em standard} if it is of the form $I = [\frac{a}{2^k},\frac{a+1}{2^k}]$. For example, the interval $[\frac 1 4,\frac 1 2]$ is dyadic and standard, whereas $[\frac 1 4,\frac 3 4]$ is dyadic but not standard. }''
interval and $\sigma\subset \R^2$ be a regular sector generated by the (oriented) basis $v_1,v_2\subset \Z^2$.   Taking $\RR_\sigma\subset \RR$ to be the set of real rays in $\sigma$, we define a homeomorphism $\phi:\RR_\sigma\to I$ as follows, identifying each rational ray with its generator $v\in\Z^2$.  We first assign facets to endpoints: $\phi(v_1) = \frac{a}{2^k}$ and $\phi(v_2) = \frac{a+1}{2^k}$.  Then we proceed inductively by averaging: $\phi(v_1+v_2) = \frac12(\phi(v_1) + \phi(v_2))$, and more generally, if $\phi$ is defined at two primitive points (i.e. integral points where the two coordinates are coprime) $v,v'\in\sigma\cap \Z^2$ but at no points in the intervening sector, then we set $\phi(v+v') = \frac12(\phi(v)+\phi(v'))$.  One checks that this procedure results in an order-preserving bijection between rational rays in $\sigma$ and dyadic points in $I$ which therefore extends to an orientation-preserving homeomorphism $\phi:\RR_\sigma \to I$.  Note that when $I=[0,1]$ and $\sigma$ is the cone generated by $(1,0)$ and $(1,1)\in\R^2$ with $\RR_\sigma$ parameterized by the slope, $\phi$ is known as \emph{Minkowski's question mark function}.  

One globalizes $\phi$ by choosing any partition $\R^2 = \sigma_1\cup\dots\cup\sigma_d$ into regular sectors and a corresponding partition $[0,1]=I_1\cup\dots\cup I_d$ into standard dyadic intervals and defining $\phi|_{\RR_{\sigma_j}}:\RR_{\sigma_j}\to I_j$ sector-wise as above.  

To see that $\phi$ conjugates piecewise linear automorphisms of $\Z^2$ to dyadic circle maps, it suffices to verify the following additional fact.  Let $\phi:\RR_\sigma \to I$ be the `local' version of $\phi$ defined above.  Let $\sigma'$ be another regular sector, $I'$ another standard dyadic interval, and $\phi':\RR_{\sigma'}\to I'$ the analogous homeomorphism.  Then $\phi'\circ F = f\circ \phi$ where $F\in \SL(2,\Z)$ is the unique element such that $F(\sigma) = \sigma'$ and $f:\R\to\R$ is the unique increasing affine map such that $f(I) = I'$.

\section{Decomposition of piecewise linear automorphisms of $\Z^2$}
\label{Sec:decomposition}

A \emph{refinement} of a fan $\Delta$ is a fan $\Delta'$ with the same support as $\Delta$ and such that for every cone $\sigma'\in\fan'$, there exists $\sigma\in\fan$ such that $\sigma'\subseteq\sigma$. The following lemmas are well-known in convex geometry, but we include the proofs for completeness.

\begin{lem}
\label{lem:regularization}
Any fan admits a regular refinement.
\end{lem}

\begin{proof}
It suffices to prove that any sector $\sigma$ may be partitioned into regular sectors.
If $\sigma$ is not regular, then the ray generators $u_1,u_2$ of its facets must satisfy $|\det[u_1\ u_2]|\ge 2$.
This implies that there must be at least one primitive point $w$ in the interior of the parallelogram formed by $u_1$ and $u_2$.
The ray $\R_+ w$ divides $\sigma$ into two new sectors.  Since $|\det[w\ u_i]| < |\det[u_1\ u_2]|$, $i=1,2$, we see that finitely many 
such subdivisions will result in a partition of $\sigma$ into regular sectors.
\end{proof}

Let $\sigma\in\fan$ be a regular sector with its boundary rays generated by $u_1,u_2\in\Z^2$.  Let $\tau\subset \sigma$ be the ray generated by some point $u\in\Z^2$ in the interior of $\sigma$.  Then $\sigma\setminus\tau$ is a union of two disjoint sectors $\sigma_1,\sigma_2$. One checks that they are both regular if and only if $u=u_1+u_2$.  We then call the fan 
$\fan'=(\fan\setminus\{\sigma\})\cup\{\sigma_1,\sigma_2,\tau\}$ the {\em simple split}
of $\fan$ at $\sigma$. If $\fan$ is regular then so is $\fan'$.  Conversely, we call $\Delta$ the {\em simple merge} of $\Delta'$ at $\sigma_1$ and $\sigma_2$.

\begin{lem}
\label{lem:regular_refinement}
Suppose that $\fan,\fan'$ are regular fans, and $\fan'$ refines $\fan$. Then one can obtain
$\fan'$ from $\fan$ by a sequence of simple splits.
\end{lem}

\begin{proof}
We need to consider only the case where $\fan$ is the fan determined by a single regular sector $\sigma := \R_+ u + \R_+ v$.  Working inductively, it suffices to show that either $\fan'$ includes the ray 
$\tau := \R_+ (u+v)$ that barycentrically subdivies $\sigma$, or $\Delta' = \Delta$.  If $\tau \notin \Delta'$, then we let $\tau_1 = \R_+(au + bv)$, 
$\tau_2 = \R_+(cu + dv)$ be the facets of the sector $\sigma' \in \Delta'$ that contains $\tau$.  We may assume that $a > b\geq 0$ and $d> c\geq 0$ are all integers.  Since $\sigma'$ is regular, we get
$$
1 = ad-bc \geq (b+1)(c+1) - bc = b+c+1.
$$
Thus $b=c = 0$ and $a=d=1$, i.e. $\sigma' = \sigma$ and $\fan' = \fan$.
\end{proof}

From now on, $F$ will always denote a piecewise linear automorphism of $\Z^2$.  Our proof of Theorem A amounts to carefully analyzing how fans transform under $F$.  Note that if $\Delta$ is a  regular fan compatible with $F$, then the image fan $F(\Delta) :=\{F(\sigma):\sigma\in \Delta\}$ is also regular.  If $\Delta'$ is another regular fan, then we say that $F$ is an {\em isomorphism} from $\Delta$ to $\Delta'$ if $F(\Delta) = \Delta'$.  Similarly, we say that $F$ is a {\em simple split} or {\em simple merge} of $\fan'$ if $F(\Delta)$ is accordingly, a simple split or simple merge of $\Delta'$.  We call $F$ a {\em simple} map from $\Delta$ to $\Delta'$ if $F$ is of one of these three types.

Regardless, we associate to $F$ a partially defined, `approximate' map $F_\sharp:\Delta\to\Delta'$ as follows: for each $\sigma\in\fan$, we take $F_\sharp\sigma$ to be the smallest cone in $\fan'$ containing $F(\sigma)$.  If no single cone in $\fan'$ contains $F(\sigma)$, then we leave $F_\sharp\sigma$ undefined.  Hence $F_\sharp\tau$ is defined for all rays $\tau\in\fan$, though the image might be a sector or a ray; whereas for each sector $\sigma\in\Delta$, the cone $F_\sharp \sigma$ is either another sector or undefined. If $G$ is another piecewise linear automorphism of $\Z^2$ and $G_\sharp:\fan'\to\fan''$ is an approximate map of fans, then we have $G_\sharp F_\sharp \sigma = (G\circ F)_\sharp\sigma$ provided that the left side is defined.

\begin{prop}
\label{prop:decomposition}
Let $\fan$ be a regular fan compatible with $F$.  Then there is a decomposition $F=f_{n-1}\circ\dots \circ f_0$ into piecewise linear automorphisms of $\Z^2$ and a corresponding sequence of  regular fans $\Delta_0,\dots,\Delta_n$ such that
\begin{itemize}
 \item $\Delta_0 = \Delta_n = \Delta$;
 \item $\Delta_j$ is compatible with $f_j$ for $0\leq j\leq n-1$;
 \item $f_j$ is a simple map from $\Delta_j$ to $\Delta_{j+1}$.
\end{itemize}
\end{prop}

\begin{proof}
First, set $f_0=F$ and $\fan_1=F(\fan)$. Lemma \ref{lem:regularization} then gives a regular common refinement $\fan'$ of $\fan$ and $F(\fan)$. Then by Lemma~\ref{lem:regular_refinement}, we know that we can obtain $\fan'$ from $F(\fan)$ by performing finitely many simple splits, i.e. we have a sequence
\[
\xymatrix{
\fan \ar[rr]^{F} && \fan_1=F(\fan) \ar[rr]^{\id} && \fan_2 \ar[rr]^{\id}
&&\cdots \ar[rr]^{\id} &&
\fan_k=\fan'
}
\]
in which all but the first arrow represents a simple split; and the first arrow is an isomorphism.  Similarly, there is a sequence
\[
\xymatrix{
\fan_k=\fan' \ar[rr]^{\id} && \fan_{k+1} \ar[rr]^{\id} &&\cdots \ar[rr]^{\id}
&& \fan_{n-1} \ar[rr]^{\id} &&
\fan = \fan_n
}
\]
such that each arrow is a simple merge, $i=k,\cdots,n-1$.  Putting the two sequences together completes the proof.
\end{proof}

\section{Proof of Theorem A}
\label{Sec:proof}
We call a cone $\tau\in\Delta$ {\em deterministic} for the approximate self-map $F_\sharp:\Delta\to\Delta$ if all its approximate forward images $(F_\sharp)^j\tau$, $j\ge 1$, are defined. 

\begin{prop}
For any piecewise linear automorphism $F$ and any fan $\fan$, there exists a regular refinement $\fan'$ of $\fan$ such that $\fan'$ is regular and compatible with $F$, and every ray in $\fan'$ is deterministic for $F_\sharp:\Delta'\to\Delta'$.
\end{prop}

\begin{proof}
Let $f_j$ and $\Delta_j$ be the maps and associated fans in the decomposition of $F$ from Proposition \ref{prop:decomposition}.  Since $\Delta_0 = \Delta_n$, we may extend these sequences of maps/fans periodically, e.g. $f_j := f_{j\mod n}$ for all $j\in\N$.  Suppose that there exists a ray $\tau\in\Delta$ that is not deterministic for $F_\sharp:\Delta\to\Delta$.  In particular, we know $F(\Delta)\neq \Delta$, so $n\geq 2$.  Then by replacing $\tau$ with an approximate forward image, we may assume that $F_\sharp \tau = \sigma$ is a sector and that $(F_\sharp)^k \sigma$ is undefined for $k\in\N$ large enough.  

We may refine this observation using the decomposition of $F$.  Namely, there exist indices $i < k \in\N$, and a ray $\tau_i\in\Delta_i$ such that
\begin{itemize}
 \item[(i)] $\sigma_{i+1} := f_{i\sharp} \tau_i \in \Delta_{i+1}$ is a sector;
 \item[(ii)] $\sigma_{j+1} := f_{j\sharp}\sigma_j \in \fan_{j+1}$ is well-defined for $ j =i+1,\cdots, k-1$, but $f_{k\sharp}\sigma_{k}$ is not defined;
 \item[(iii)] $i,\tau_i,k$ are chosen so that $k-i$ is as small as possible given the first two conditions. 
\end{itemize}
That is, we attend to the following sequence:
\[
\xymatrix{
\tau_i\in\fan_{i} \ar[rr]^{f_i} && \sigma_{i+1}\in\fan_{i+1} \ar[rr]^{f_{i+1}} &&\cdots \ar[rr]^{f_{k-1}} &&
\sigma_{k}\in\fan_k  \ar[rr]^{f_k} &&
\fan_{k+1} 
}
\]
The idea is to inductively refine the fans $\fan_{i+1},\cdots,\fan_{k}$ so that $f_k$ becomes an isomorphism.
 
Condition (i) together with the fact that $f_i$ is simple from $\Delta_i$ to $\Delta_{i+1}$ means that $\tau_{i+1} := f_i(\tau_i)$ is the unique ray that divides $\sigma_{i+1}$ into two regular cones.  Replacing $\sigma_{i+1}$ with these two cones and adding the ray $\tau_{i+1}$ to $\Delta_{i+1}$ turns the map $f_i:\Delta_i\to\Delta_{i+1}$ into an isomorphism from $\Delta_i$ to $\Delta_{i+1}$.  Note that in this (and succeeding) steps, we apply the same refinement to all fans $\Delta_{i+\ell n}$ equivalent mod $n$ to $\Delta_i$.  Since $n\geq 2$ this does not affect the outcome.

Turning from $f_i$ to $f_{i+1}$, we find two possible cases.  If $i+1=k$, then $f_{i+1}(\sigma_{i+1})$ is equal to the union of two adjacent sectors in $\Delta_{i+2}$.  Since the above refinement of $\Delta_{i+1}$ split $\sigma_{i+1}$ into two regular cones, $f_{i+1}$ is now an isomorphism from $\Delta_{i+1}$ to $\Delta_{i+2}$. 

Otherwise, $i+1<k$.  If $f_{i+1}$ was initially a merge, then there is another ray $\tilde\tau_{i+1}\in\Delta_{i+1}$ such that $f_{i+1}(\tilde\tau_{i+1})\notin\Delta_{i+2}$.  However, our minimality assumption on $k-i$ guarantees that the images $f_{i+1}(\tilde\tau_{i+1})$ and $f_{i+1}(\tau_{i+1})$ lie in different sectors.  Hence regardless of whether $f_{i+1}$ is a merge or not, $f_{i+1}(\tau_{i+1})$ divides $\sigma_{i+2}$ into two \emph{regular} cones.  So as with $\Delta_{i+1}$, we can refine $\Delta_{i+2}$ by splitting $\sigma_{i+2}$ into two regular sectors, the result being that $f_{i+1}$ remains a simple map.  In this case, we move on to $f_{i+2}$ and repeat the process until we reach $f_k$.  In the end, $f_k$ is improved from a split to an isomorphism, and maps $f_j$ with $j\not\equiv k\mod n$ remain simple of the same type as before.


If there remain rays in $\Delta=\Delta_0$ that are not deterministic for $F_\sharp$, then we repeat the process.  Since there are only finitely many $f_j$ (modulo $n$), finitely many repetitions leads to a situation in which either all rays in $\Delta$ are deterministic or all $f_j$ are isomorphisms.  In the latter case $F$ is itself an isomorphism from $\Delta$ to $\Delta$, and all rays are deterministic anyhow.
\end{proof}

Finally, we can now prove the first part of Theorem A.

\begin{proof}[Proof of Theorem A, part I]
Given any piecewise linear automorphism $F$ of $\Z^2$,
we first find any fan $\fan$ such that $\fan$ is compatible
with $F$. Next, we find a regular refinement $\fan'$ of $\fan$, then $\fan'$ is still compatible with $F$.
Finally, we find a further refinement $\fan''$ such that every ray in $\fan''$
is deterministic for $F_\sharp$. Observe how $F$ acts
on the rays and sectors of $\fan''$, either one of the following will happen:

\begin{itemize}
\item There is never a ray mapped into a sector, so rays always map to rays, and $F$ is permuting the rays. In this case, some iterate of every ray is going back to itself and $F$ is of finite order.

\item There is a ray maps into a sector $\sigma$, then $\sigma$ must be deterministic for $F_\sharp$, and $(F_\sharp)^j \sigma$, $j\ge 1$, 
are all defined and are sectors. Then again, we have finitely many sectors, there must be some $i<j$ such that
$(F_\sharp)^i \sigma=(F_\sharp)^j \sigma$. As a consequence, there is a ray in $(F_\sharp)^i \sigma$ that is fixed by $F^{j-i}$,
so the circle map induced by $F$ has a periodic point.
\end{itemize}

In either case, the rotation number of $F$ is rational because there is a periodic point.
This completes the proof.
\end{proof}

\begin{rem}
Our proof also gives an algorithm for finding the rotation number of a piecewise linear automorphism $F$. First, find a fan $\fan$ that is compatible with $F$. Then find a regular refinement $\fan'$ of $\Delta$ such that every ray in $\fan'$ is deterministic for $F_\sharp$. This can be done effectively by
Lemma~\ref{lem:regularization}, Lemma~\ref{lem:regular_refinement}, and Proposition~\ref{prop:decomposition}.  Tracing the orbit of each ray in $\Delta'$, as we did in the above proof, will locate the sectors containing periodic (real) rays, and this will then give the rotation number of $F$.
\end{rem}

\begin{rem}
We stress here that in the definition of piecewise linear automorphism,
the condition $F(\Z^2)=\Z^2$ is crucial for the Theorem A to hold. If we only require that $F(\Z^2)\subseteq\Z^2$,
then it is easy to find a linear map, e.g.,
$F=\left[\begin{smallmatrix} 4 & -3 \\ 3 & 4 \end{smallmatrix}\right]$
such that the corresponding circle homeomorphism has an irrational rotation number.
\end{rem}

The second part of Theorem A can be proved by a concrete construction.

\begin{proof}[Proof of Theorem A, part II]
It suffices to show that for any given integer $n\ge 3$, we can find a piecewise
linear automorphism $F$ of $\Z^2$ such that $\rho(F)=1/n$.
For each integer $n\ge 3$ we can find a regular fan $\fan$ with $n$ sectors,
say $\sigma_0,\sigma_1,\cdots,\sigma_n=\sigma_0$, arranged in the counterclockwise order.
Since each $\sigma_i$ is regular, there is a unique $L_i\in \SL_2(\Z)$ such that $L_i$
maps $\sigma_i$ bijectively onto $\sigma_{i+1}$, $i=0,\cdots,n-1$ and maps the lattice
points in $\sigma_i$ bijectively onto lattice points in $\sigma_{i+1}$.
Define $F$ as $F|_{\sigma_i}=L_i|_{\sigma_i}$, i.e., $F$ is rotating the sectors,
sending $\sigma_i$ to $\sigma_{i+1}$. Then $F$ is periodic of period $n$, and therefore
$\rho(F)=1/n$.
\end{proof}


\end{document}